\documentclass[reqno]{amsart}
\usepackage{amssymb, amsmath, amsthm, amscd, mathrsfs}

\usepackage{paralist}
\usepackage{cite}
\usepackage{bigints}
\usepackage{dsfont}
\usepackage{empheq}
\usepackage{amssymb}
\usepackage{cases}
\usepackage{enumitem}
\allowdisplaybreaks

\usepackage{hyperref}

\theoremstyle{plain}
\newtheorem{theorem}{Theorem}[section]

\newtheorem{proposition}{Proposition}
\theoremstyle{definition}
\newtheorem{definition}[theorem]{Definition}
\newtheorem{remark}{Remark}

\def \dv {\mathrm{div}}
\def \d {\mathrm{d}}

\title[Logarithmic stability estimates for initial data] 
      {Logarithmic stability estimates for initial data in Ornstein-Uhlenbeck equation on $L^2$-space}

\author{S. E. Chorfi}
\author{L. Maniar}

\address{S. E. Chorfi, L. Maniar, Cadi Ayyad University, Faculty of Sciences Semlalia, LMDP, UMMISCO (IRD-UPMC), B.P. 2390, Marrakesh, Morocco}

\email{s.chorfi@uca.ac.ma, maniar@uca.ma}

\makeatletter 
\@namedef{subjclassname@2020}{%
  \textup{2020} Mathematics Subject Classification}
\makeatother


\begin{document}
\begin{abstract}
In this paper, we continue the investigation on the connection between observability and inverse problems for a class of parabolic equations with unbounded first order coefficients. We prove new logarithmic stability estimates for a class of initial data in the Ornstein-Uhlenbeck equation posed on $L^2\left(\mathbb{R}^N\right)$ with respect to the Lebesgue measure. The proofs combine observability and logarithmic convexity results that include a non-analytic semigroup case. This completes the picture of the recent results obtained for the analytic Ornstein-Uhlenbeck semigroup on $L^2$-space with invariant measure.
\end{abstract}
\maketitle

\section{Introduction}
Let $N\ge 1$ be an integer and let $\theta >0$ be a fixed time. We consider the Ornstein-Uhlenbeck equation given by
\begin{empheq}[left =\empheqlbrace]{alignat=2}
\begin{aligned}
&\partial_t u = \Delta u + B x\cdot \nabla u, \qquad 0<t<\theta , && x\in \mathbb{R}^N, \\
& u\rvert_{t=0}=u_0(x),   && x \in \mathbb{R}^N, \label{e1}
\end{aligned}
\end{empheq}
where $B$ is a real constant $N\times N$-matrix, $u_0 \in L^2\left(\mathbb{R}^N\right)$ is an unknown initial datum, and the
dot sign denotes the inner product in $\mathbb{R}^N$.

In the present paper, we are concerned with the following problem:

\noindent\textbf{Inverse initial data problem.} Determine the unknown initial datum $u_0$ belonging to an admissible set $\mathcal{I}$ from the measurement
$$u\rvert_{(0,\theta) \times \omega},$$
where $\omega \subset \mathbb{R}^N$ is an observation region that will be given later. Two problems to be answered:
\begin{itemize}
\item[$\bullet$] Uniqueness: for two initial data $u_0$ and $u_1$, does the equality $u(u_0)=u(u_1)$ in $\left(0, \theta\right)\times \omega$ imply $u_0=u_1$ in $L^2\left(\mathbb{R}^N\right)$?
\item[$\bullet$] Stability: is it possible to estimate $\|u_0-u_1\|_{L^2\left(\mathbb{R}^N\right)}$ by a suitable norm of the quantity $\left(u(u_0)-u(u_1)\right)\rvert_{\left(0, \theta\right)\times \omega}$?
\end{itemize}
In the parabolic framework, one expects a stability estimate of logarithmic type, i.e.,
\begin{equation} \label{lsi}
\|u_0\|_{L^2\left(\mathbb{R}^N\right)} \le \frac{c}{|\log \|u\||^\alpha}
\end{equation}
for a constant $c>0$, some $\alpha \in (0,1]$, a suitable norm of $u\rvert_\omega$, typically $\|u\|_{L^2\left(0, \theta ; L^{2}(\omega)\right)}$ or $\|u\|_{H^1\left(0, \theta ; L^{2}(\omega)\right)}$, and for all $u_0 \in \mathcal{I}$.

The estimate \eqref{lsi} is commonly referred to as conditional stability since it depends on the admissible set $\mathcal{I}$. Note that the rate of conditional stability depends on the choice of the admissible set. Also, conditional stability is very useful when dealing with the numerical reconstruction of initial data, see for instance \cite{LYZ'09, YZ'01}.

The Ornstein-Uhlenbeck equation \eqref{e1} appears in several fields of research and applications including stochastic processes, Malliavin calculus, theory of quantum fields as well as control theory. We refer to \cite{BP'18, CFMP'05, LRM'16, Lu'97, LMP'20, Me'01, MPP'02, MPRS'03} and the references therein for more details. The associated Ornstein-Uhlenbeck semigroup has been extensively studied in two spaces; the space $L^2\left(\mathbb{R}^N, \d x\right)$ (with the Lebesgue measure, shortly denoted by $L^2\left(\mathbb{R}^N\right)$) and in $L^2\left(\mathbb{R}^N, \d \mu\right)$ with the so-called invariant measure $\mu$. The existence of $\mu$ is equivalent to the fact that the spectrum of the matrix $B$ lies in the left half-plane, i.e.,
$$\sigma(B)\subset \mathbb{C}_-:=\{\lambda \in \mathbb{C} \colon \mathrm{Re}\, \lambda <0\}.$$
The semigroup enjoys different properties with respect to the previous spaces. It is a $C_0$-semigroup on $L^2\left(\mathbb{R}^N, \d x\right)$ which is not analytic (it is not an eventually norm continuous semigroup), unless $B=0$, whereas its realization in $L^2\left(\mathbb{R}^N, \d \mu\right)$ is an analytic $C_0$-semigroup with angle depending on the matrix $B$.

As for inverse problems, parabolic equations with bounded coefficients have been extensively studied, see for instance \cite{KL'06, KT'07, LYZ'09, YZ'08}. Surprisingly, there are only a few researches for parabolic equations with unbounded coefficients, especially for the Ornstein-Uhlenbeck equation. We are only aware of the paper by Lorenzi \cite{Lo'11}, where the author studied the identification (existence and uniqueness) of a constant $\alpha \in \mathbb{R}$ such that $B=\alpha B_0$ in \eqref{e1}, where the matrix $B_0$ and some prescribed data are known. This inverse problem was motivated by the determination of the matrix $B$ since its properties play a crucial role in the study of the Ornstein-Uhlenbeck equation.

Inverse problems are closely related to controllability and observability questions. The link between the two theories was initially observed in the case of a hyperbolic system in \cite{PY'96}. Then many stability estimates have been proven for inverse problems by using observability inequalities. Such inverse problems for parabolic systems are highly ill-posed in the sense of Hadamard, so one cannot expect strong stability in general. In this context, the logarithmic convexity method has been successfully applied to prove logarithmic stability for initial data. Many works have been done for second-order parabolic equations with bounded coefficients \cite{KL'06, KT'07, XY'00, YZ'08}, and more generally for some abstract systems governed by analytic semigroups of angle $\frac{\pi}{2}$ as in the self-adjoint case; see \cite{GaT'11}.

We emphasize that the logarithmic stability for initial data in a parabolic equation with \textit{unbounded coefficients} has not been studied before, as far as we know. The prototype of such equations is given by the Ornstein-Uhlenbeck equation on $L^2$-spaces, which contains an unbounded gradient term. Recently, the authors of \cite{ACM'22} have investigated the case of analytic semigroups and established an abstract stability result. Consequently, the stability problem for the Ornstein-Uhlenbeck equation has been recently resolved on $L^2\left(\mathbb{R}^N, \d \mu\right)$ with the invariant measure when $\sigma(B) \subset \mathbb{C}_-$. Therein, the unbounded gradient coefficient is handled thanks to the exponential decay of the measure density. In that case, the semigroup is analytic and the angle of analyticity is generally smaller than $\frac{\pi}{2}$. The developed approach does not work on $L^2\left(\mathbb{R}^N\right)$ due to the lack of analyticity of the semigroup. The logarithmic convexity estimate has been initially known for self-adjoint operators since the work of Agmon and Nirenberg \cite{AN'63}, and has been generalized to analytic semigroups in \cite{KP'60} (see also \cite{Mi'75}). In our non-analytic framework, we exploit the explicit representation formula of the semigroup to prove a logarithmic convexity estimate, and then deduce a logarithmic stability for initial data using final state observability with respect to the $L^2\left(\mathbb{R}^N\right)$-norm.

We mainly aim to clarify the relations between the final state observability, logarithmic convexity and stability estimates for initial data in the Ornstein-Uhlenbeck equation. Such a relation is by now fully understood in the analytic framework of parabolic equations with bounded coefficients. We start this article by gathering some useful facts on the Ornstein-Uhlenbeck equation (Section \ref{sec2}). In Section \ref{sec3}, we first prove a logarithmic convexity estimate for the Ornstein-Uhlenbeck equation \eqref{e1} exploiting the explicit form of the associated semigroup. Then we recall some recent results concerning the final state observability of \eqref{e1}. Afterward in Section \ref{sec4}, we show that the logarithmic convexity along with the observability imply some logarithmic stability estimates for certain classes of initial data. These results allow us in particular to extend some known results for the heat equation on $\mathbb{R}^N$. Finally, Section \ref{sec5} is devoted to some final comments and open problems related to the stability estimates for initial data in parabolic equations with unbounded coefficients.

\section{Miscellaneous facts on the Ornstein-Uhlenbeck semigroup} \label{sec2}
We summarize some useful facts about the Ornstein-Uhlenbeck semigroup on $L^2\left(\mathbb{R}^N\right)$ for future use. We refer to \cite{Me'01,MPRS'03} and the references therein for detailed proofs.

Denote by $I$ the identity operator and by $\|\cdot\|$ the standard norm in $L^{2}\left(\mathbb{R}^{N}\right)$. The Ornstein-Uhlenbeck operator is the sum of the diffusion term and the drift term.

\subsection{The diffusion part}
Let us first consider the diffusion term. It is well known that the Laplace operator $\Delta$ on $L^{2}\left(\mathbb{R}^N\right)
$ with maximal domain
\begin{align*}
D(\Delta) &:=\left\{u \in L^{2}\left(\mathbb{R}^{N}\right): \Delta u \in L^{2}\left(\mathbb{R}^{N}\right)\right\}\\
&= H^{2}\left(\mathbb{R}^N\right),
\end{align*}
is self-adjoint and generates the heat $C_0$-semigroup which is explicitly defined by $U(t) \colon L^2\left(\mathbb{R}^N\right) \rightarrow L^2\left(\mathbb{R}^N\right)$,
\begin{align*}
U(0) &=I,\\
(U(t) f)(x)&=\frac{1}{(4 \pi t)^{N/2}} \int_{\mathbb{R}^{N}} \mathrm{e}^{-\frac{|x-y|^2}{4 t}} f(y) \,\d y, \qquad t>0, \quad x\in \mathbb{R}^N
\end{align*}
for every $f\in L^2\left(\mathbb{R}^N\right)$, where $|\cdot|$ denotes the Euclidean norm on $\mathbb{R}^N$. We clearly see that
\begin{equation*}
U(t) f=h_{t} * f,
\end{equation*}
where
\begin{equation*}
h_{t}(y)=\frac{1}{(4 \pi t)^{N/2}} \mathrm{e}^{-\frac{|y|^2}{4 t}}, \quad t>0, \quad y\in \mathbb{R}^N.
\end{equation*}
Note that the heat semigroup $\left(U(t)\right)_{t \ge 0}$ is analytic and contractive on $L^{2}\left(\mathbb{R}^N\right)$.

\subsection{The drift part}
Let $B$ be the drift matrix in \eqref{e1}. We consider the drift operator
$$
\mathcal{L}=B x \cdot \nabla=\sum_{i, j=1}^{N} b_{i j} x_{j} \partial_{x_i},
$$
with its maximal domain
$$
D(\mathcal{L})=\left\{u \in L^{2}\left(\mathbb{R}^{N}\right): \mathcal{L} u \in L^{2}\left(\mathbb{R}^{N}\right)\right\},
$$
where $\mathcal{L} u$ is understood in the sense of distributions. The operator $\left(\mathcal{L}, D(\mathcal{L})\right)$ is closed in $L^{2}\left(\mathbb{R}^{N}\right)$ and the space $C_{c}^{\infty}(\mathbb{R}^N)$ is a core of $\mathcal{L}$. The operator $\mathcal{L}$ generates a $C_{0}$-group $\left(S(t)\right)_{t\in \mathbb{R}}$ given by
\begin{equation}
\left(S(t) f\right)(x)=f\left(\mathrm{e}^{t B} x\right), \quad t\in \mathbb{R}, \quad x\in \mathbb{R}^N
\end{equation}
for $f \in L^{2}\left(\mathbb{R}^{N}\right)$. Moreover,
\begin{equation} \label{edn}
\|S(t) f\|=\mathrm{e}^{-\frac{t}{2} \operatorname{tr}(B)}\|f\|, \qquad t\in \mathbb{R} 
\end{equation}
for every $f \in L^{2}\left(\mathbb{R}^{N}\right)$.

\subsection{The Ornstein-Uhlenbeck semigroup}
Let $B^*$ denote the transpose matrix of $B$. We introduce the matrices
$$
Q_{t}=\int_{0}^{t} \mathrm{e}^{s B} \, \mathrm{e}^{s B^{*}} \,\d s, \qquad t>0,
$$
which are symmetric and positive definite ($Q_0$ is the null matrix). The Ornstein-Uhlenbeck semigroup $T(t) \colon L^2\left(\mathbb{R}^N\right) \rightarrow L^2\left(\mathbb{R}^N\right)$ is given by Kolmogorov's formula
\begin{align*}
T(0) &=I,\\
(T(t) f)(x)&=\frac{1}{\sqrt{(4 \pi)^{N} \operatorname{det} Q_{t}}} \int_{\mathbb{R}^{N}} \mathrm{e}^{-\frac{1}{4}\left\langle Q_{t}^{-1} y, y\right\rangle} f\left(\mathrm{e}^{t B} x-y\right) \d y, \quad t>0, \quad x\in \mathbb{R}^N
\end{align*}
for every $f\in L^2\left(\mathbb{R}^N\right)$. This can be written as
\begin{equation}\label{eouf}
T(t) f=S(t)\left(g_{t} * f\right),
\end{equation}
where
\begin{equation}
g_{t}(y)=\frac{1}{\sqrt{(4 \pi)^{N} \operatorname{det} Q_{t}}} \mathrm{e}^{-\frac{1}{4}\left\langle Q_{t}^{-1} y, y\right\rangle}, \quad t>0, \quad y\in \mathbb{R}^N.
\end{equation}
Let us define the maximal domain of $\mathcal{A}$ by
$$
D\left(\mathcal{A}\right)=\left\{u \in L^{2}\left(\mathbb{R}^{N}\right) \cap H_{\mathrm{loc}}^{2}\left(\mathbb{R}^{N}\right) : \mathcal{A} u \in L^{2}\left(\mathbb{R}^{N}\right)\right\}.
$$
The space $C_{c}^{\infty}\left(\mathbb{R}^{N}\right)$ is a core for $\mathcal{A}$. More precisely, $D\left(\mathcal{A}\right)$ is characterized by
$$
D\left(\mathcal{A}\right)=\left\{u \in H^{2}\left(\mathbb{R}^{N}\right) : B x\cdot \nabla u \in L^{2}\left(\mathbb{R}^{N}\right)\right\}=D(\Delta) \cap D(\mathcal{L}).
$$
It is known that $\left(\mathcal{A},D\left(\mathcal{A}\right)\right)$ generates a $C_0$-semigroup on $L^{2}\left(\mathbb{R}^{N}\right)$ given by $\left(T(t)\right)_{t\ge 0}$ which satisfies the estimate
$$
\|T(t)\| \leq \mathrm{e}^{-\frac{t}{2} \operatorname{tr}(B)}, \qquad t\ge 0,
$$
where $\|\cdot\|$ also denotes the operator norm. The $C_0$-semigroup $\left(T(t)\right)_{t\ge 0}$ is not analytic in $L^{2}\left(\mathbb{R}^{N}\right)$ (e.g. it is not an eventually norm continuous semigroup), except for $B=0$.

\section{Logarithmic convexity and observability} \label{sec3}
\subsection{Logarithmic convexity estimates}
The logarithmic convexity method is one of the well-known approaches that had been widely used to prove conditional stability for improperly posed problems such as backward parabolic equations as well as inverse initial data problems. The interested reader can be referred to \cite{AN'63, ACM'21,Is'17,KP'60,Pa'75} and the references therein.

Next we give an abstract definition to a logarithmic convexity estimate. Let $\left(H, \langle \cdot, \cdot \rangle\right)$ be a Hilbert space with corresponding norm $\|\cdot\|$. Let $\left(\mathrm{e}^{t A}\right)_{t \ge 0}$ be a $C_0$-semigroup on $H$ associated with its generator $A \colon D(A)\subset H\rightarrow H$. Inspired by \cite{AN'63, KP'60}, we introduce the following definition.
\begin{definition}
We say that the $C_0$-semigroup $\left(\mathrm{e}^{t A}\right)_{t \ge 0}$ satisfies a logarithmic convexity estimate for $\theta >0$ if there exists a constant $\kappa_\theta \ge 1$ and a function $w \colon (0,\theta) \rightarrow (0,1)$, $w(0)=0$ and $w(\theta)=1$, so that the following estimate holds
\begin{equation}\label{elc}
\left\|\mathrm{e}^{t A} u\right\| \le \kappa_\theta \|u\|^{1-w(t)} \left\|\mathrm{e}^{\theta A} u \right\|^{w(t)}
\end{equation}
for all $t\in [0,\theta]$ and all $u\in H$.
\end{definition}
It is known that the logarithmic convexity holds for self-adjoint bounded above operators with $\kappa_\theta=1$ and $w(t)=\dfrac{t}{\theta}$, see e.g. \cite[Section 2]{GaT'11}. If the operator is subordinated to its symmetric part, there is a logarithmic convexity result in \cite[Theorem 3.1.3]{Is'17}, where $w(t)=\dfrac{1-\mathrm{e}^{-ct}}{1-\mathrm{e}^{-c\theta}}$ for some constant $c>0$. More generally, it holds for analytic semigroups with a more general function $w(t)$. We refer to \cite{ACM'22} for more details.

\begin{remark}
Note that the estimate \eqref{elc} implies the backward uniqueness for the semigroup $\left(\mathrm{e}^{t A}\right)_{t \ge 0}$, namely the following property: if $\mathrm{e}^{\theta A} u=0$ for some $u\in H$, then $u=0$. Furthermore, the logarithmic convexity can be seen as a stability estimate for the backward uniqueness. More precisely, if one assumes a priori bound on the initial data $\|u\| \le R$ for some positive constant $R$, the norm of the corresponding solution $\left\|\mathrm{e}^{t A} u\right\|$, $t\in (0,\theta)$ is small whenever $\left\|\mathrm{e}^{\theta A} u \right\|$ is small.
\end{remark}

Since the Laplace operator with domain $H^2\left(\mathbb{R}^N\right)$ is self-adjoint and negative on $L^2\left(\mathbb{R}^N\right)$, the heat semigroup $\left(U(t)\right)_{t\ge 0}$ satisfies the logarithmic convexity
$$\|U(t) f\| \le \|f\|^{1-\frac{t}{\theta}} \|U(\theta) f\|^{\frac{t}{\theta}}, \qquad f\in L^2\left(\mathbb{R}^N\right), \qquad t\in [0,\theta].$$
In fact, the function $t \mapsto \|U(t) f\|^2$ is even log-convex in this case. The same holds for the drift $C_0$-group $\left(S(t)\right)_{t\in \mathbb{R}}$. Indeed, using \eqref{edn}, we see that
$$\|S(t) f\| = \|f\|^{1-\frac{t}{\theta}} \|S(\theta) f\|^{\frac{t}{\theta}}, \qquad f\in L^2\left(\mathbb{R}^N\right), \qquad t\in [0,\theta].$$
Next we prove a logarithmic convexity estimate for the Ornstein-Uhlenbeck semigroup $\left(T(t)\right)_{t\ge 0}$.
\begin{proposition}\label{proplc}
There exist constants $c=c(\theta)\in (0,1]$ and $\kappa_\theta \ge 1$ such that the following estimate holds
\begin{equation}\label{elc0}
\|T(t) f\| \le \kappa_\theta \|f\|^{1-c\frac{t}{\theta}} \|T(\theta)f \|^{c\frac{t}{\theta}}, \qquad f\in L^2\left(\mathbb{R}^N\right), \qquad t\in [0,\theta].
\end{equation}
\end{proposition}

\begin{proof}
Using \eqref{edn} and \eqref{eouf}, the estimate \eqref{elc0} is implied by the following inequality
\begin{equation}\label{elc1}
\|g_t * f\| \le \|f\|^{1-c\frac{t}{\theta}} \|g_\theta * f \|^{c\frac{t}{\theta}}, \qquad f\in L^2\left(\mathbb{R}^N\right), \qquad t\in (0,\theta],
\end{equation}
with $\kappa_\theta=\mathrm{e}^{\frac{|\mathrm{tr}(B)|}{2}(1-c)\theta}$. Invoking the Fourier transform denoted by $\widehat{f}$ for any $f\in L^{2}\left(\mathbb{R}^{N}\right)$, we obtain the identities $
\widehat{g}_{t}(\xi)=\mathrm{e}^{-\left\langle Q_{t} \xi, \xi\right\rangle}$ and $\widehat{g_{t} * f}(\xi)=\mathrm{e}^{-\left\langle Q_{t} \xi, \xi\right\rangle} \widehat{f}(\xi)$, $\xi \in \mathbb{R}^N$. We first claim that there exists a constant $c_\theta >0$ such that
\begin{equation}\label{eqq}
\left\langle Q_{t} \xi, \xi\right\rangle \ge c_\theta \frac{t}{\theta} \left\langle Q_{\theta} \xi, \xi\right\rangle \qquad \forall t\in [0,\theta], \quad \forall \xi\in \mathbb{R}^N.
\end{equation}
Indeed, since $\frac{Q_t}{t} \to I_N$ (the identity matrix) as $t\to 0$, by continuity and positivity of the function
$$\beta(t,\xi)=
\begin{cases}
\displaystyle \frac{1}{t}\left\langle Q_{t} \xi, \xi\right\rangle &\mbox{ if } t>0,\\
\left|\xi\right|^{2} &\mbox{ if } t=0,
\end{cases}$$
in the compact set $[0,\theta]\times \mathbb{S}^{N-1}$ ($\mathbb{S}^{N-1}=\{\xi \in \mathbb{R}^N: |\xi|=1 \}$), there exist positive constants $c_1$ and $c_2$ (dependent on $\theta$) such that
$$c_1 t \le \left\langle Q_{t} \xi, \xi\right\rangle \le c_2 t \qquad \forall t\in [0,\theta],\quad \forall \xi\in \mathbb{S}^{N-1}.$$
Thus,
$$c_1 t |\xi|^2 \le \left\langle Q_{t} \xi, \xi\right\rangle \le c_2 t |\xi|^2 \qquad \forall t\in [0,\theta],\quad \forall \xi\in \mathbb{R}^N.$$
This inequality implies \eqref{eqq} with $c:=c_\theta=\frac{c_1}{c2}\le 1$. The inequality \eqref{eqq} entails that
$$\mathrm{e}^{-2\left\langle Q_{t} \xi, \xi\right\rangle} |\widehat{f}(\xi)|^2 \le |\widehat{f}(\xi)|^{2\left(1-c\frac{t}{\theta}\right)} \left(\mathrm{e}^{-2\left\langle Q_{\theta} \xi, \xi\right\rangle} |\widehat{f}(\xi)|^2\right)^{c\frac{t}{\theta}}.$$
Let $t \in (0,\theta)$ be fixed. Applying the Hölder inequality for $p=\frac{\theta}{\theta-ct}$, $q=\frac{\theta}{ct}$ and the functions $F(\xi)=|\widehat{f}(\xi)|^{2\left(1-c\frac{t}{\theta}\right)}$, $G(\xi)=\left(\mathrm{e}^{-2\left\langle Q_{\theta} \xi, \xi\right\rangle} |\widehat{f}(\xi)|^2\right)^{c\frac{t}{\theta}}$, we obtain \eqref{elc1}. This completes the proof of \eqref{elc0}.
\end{proof}

\begin{remark}
Proposition \ref{proplc}, aside from being of independent interest, implies the backward uniqueness property for the Ornstein-Uhlenbeck semigroup. This property plays an important role in some control problems.
\end{remark}

\subsection{Observability estimate}
In this subsection, we discuss some recent results on the final state observability for the system \eqref{e1}. We refer to \cite{TW'09, Za'20} for the general theory.

In \cite{BP'18}, the authors have shown the final state observability of the system \eqref{e1} at any positive time $\theta$ whenever the observation region $\omega \subset \mathbb{R}^N$ is a nonempty open set satisfying the geometric condition
\begin{equation}\label{obsreg1}
\exists \delta, r>0, \forall y \in \mathbb{R}^N, \exists y' \in \omega, \quad B\left(y', r\right) \subset \omega \text { and }\left|y-y'\right|<\delta.
\end{equation}
This condition was known as a sufficient condition ensuring the final state observability of the heat equation on $\mathbb{R}^N$ at any positive time \cite{Mi'05}. The later property has been recently characterized by the thickness of the set $\omega$ \cite{EV'18, WWZZ'19}:
\begin{definition}
Let $\gamma \in (0,1]$ and $a=\left(a_{1}, \ldots, a_{N}\right) \in\left(\mathbb{R}_{+}^{*}\right)^{N}$. Let us denote by $\mathcal{C}=\left[0, a_{1}\right] \times \ldots \times\left[0, a_{N}\right]$.
\begin{itemize}
\item A measurable set $\omega \subset \mathbb{R}^{N}$ is said to be $(\gamma, a)$-thick if
$$
\left|\omega \cap \left(x+\mathcal{C}\right)\right| \geq \gamma \prod_{j=1}^{N} a_{j} \qquad \forall x \in \mathbb{R}^{N},  
$$
where $\left|E\right|$ denotes the Lebesgue measure of a measurable set $E \subset \mathbb{R}^{N}$.
\item A measurable set $\omega \subset \mathbb{R}^{N}$ is thick if there exist $\gamma \in (0,1]$ and $a \in\left(\mathbb{R}_{+}^{*}\right)^{N}$ such that $\omega$ is $(\gamma, a)$-thick.
\item A measurable set $\omega \subset \mathbb{R}^{N}$ is $\gamma$-thick at scale $L>0$ if $\omega$ is $(\gamma, a)$-thick and $a=(L, \ldots, L) \in\left(\mathbb{R}_{+}^{*}\right)^{N}$.
\end{itemize}
\end{definition}
We emphasize that the notion of thickness is weaker than the condition \eqref{obsreg1}. Furthermore, the thickness of the observation set $\omega$ turns out to be a sufficient condition that ensures the final state observability of several parabolic equations at any positive time, including the system \eqref{e1}.

Next we state the observability inequality of system \eqref{e1} from thick sets.
\begin{proposition}
Let $\theta>0$ be fixed and let $\omega \subset \mathbb{R}^N$ be a thick set. Consider $u$ the mild solution of \eqref{e1}. Then there exists a positive constant $\kappa_\theta$ such that for all $u_0 \in L^{2}\left(\mathbb{R}^{N}\right)$, we have
\begin{equation}\label{obsineq}
\|u(\theta,\cdot)\|_{L^{2}\left(\mathbb{R}^{N}\right)}^2 \leq \kappa_\theta^2 \int_0^\theta \|u(t,\cdot)\|_{L^{2}(\omega)}^2\,\d t.
\end{equation}
\end{proposition}
The above result has been recently proven in \cite{AB'20} for the possibly degenerate fractional Ornstein-Uhlenbeck equation under a Kalman rank condition. It improves its counterpart in \cite{BP'18} regarding the Ornstein-Uhlenbeck equation for observation sets $\omega$ satisfying \eqref{obsreg1} which has been considered in \cite{ACM'22}.

\section{Logarithmic stability for a class of initial data} \label{sec4}
We introduce the set of admissible initial data:
$$\mathcal{I}_{R}=\left\{u_0 \in D(\mathcal{A}) \colon \|u_0\|_{D(\mathcal{A})} \leq R\right\}$$
for a fixed constant $R>0$.

\begin{theorem}\label{thmlogstab1}
Let $\theta>0$ be fixed and let $\omega \subset \mathbb{R}^N$ be a thick set. There exist positive constants $C$ and $C_1$ depending on $(N, \theta, \omega, R)$ such that, for all $u_0 \in \mathcal{I}_R$,
\begin{equation}
\|u_0\|_{L^{2}\left(\mathbb{R}^N\right)} \leq \frac{-C}{\log \left(C_1\|u\|_{H^1\left(0, \theta ; L^{2}(\omega)\right)}\right)}
\end{equation}
for $\|u\|_{H^1\left(0, \theta ; L^{2}(\omega)\right)}$ sufficiently small, where $u$ is the solution of system \eqref{e1}.
\end{theorem}

\begin{proof}
Let $z=u_t$ and apply \eqref{obsineq} to $z$, we obtain
\begin{equation}
\|z(\theta, \cdot)\|_{L^{2}\left(\mathbb{R}^N\right)} \leq C\|z\|_{L^{2}\left(0, \theta ; L^{2}(\omega)\right)}. \label{in1}
\end{equation}
Applying the logarithmic convexity estimate to $z$ (Proposition \ref{proplc}), we have
\begin{equation}
\left\|z(t, \cdot)\right\|_{L^{2}\left(\mathbb{R}^N\right)} \leq \kappa_\theta R^{1-c\frac{t}{\theta}}\left\|z(\theta, \cdot)\right\|_{L^{2}\left(\mathbb{R}^N\right)}^{c\frac{t}{\theta}}, \quad 0 \leq t \leq \theta. \label{in2}
\end{equation}
Since
$$u(0, \cdot)=-\int_{0}^{\theta} z(\tau, \cdot) \d \tau + u(\theta, \cdot),$$
by \eqref{in1} and \eqref{in2} we have
\begin{align*}
\|u(0, \cdot)\|_{L^{2}\left(\mathbb{R}^N\right)} &\leq \int_{0}^{\theta}\left\|z(\tau, \cdot)\right\|_{L^{2}\left(\mathbb{R}^N\right)} \d \tau + \|u(\theta, \cdot)\|_{L^{2}\left(\mathbb{R}^N\right)} \\
&\leq C \int_{0}^{\frac{\theta}{c}}\left\|z(\theta, \cdot)\right\|_{L^{2}\left(\mathbb{R}^N\right)}^{\frac{c\tau}{\theta}} \d \tau + C\|u\|_{L^{2}\left(0, \theta ; L^{2}(\omega)\right)} \quad (c\in (0,1])\\
&\leq C \frac{\theta}{c} \frac{\|z(\theta, \cdot)\|_{L^{2}\left(\mathbb{R}^N\right)}-1}{\log \|z(\theta, \cdot)\|_{L^{2}\left(\mathbb{R}^N\right)}} + C\|u\|_{L^{2}\left(0, \theta ; L^{2}(\omega)\right)}\\
& \le C \left(\frac{E-1}{\log E} + E \right),\\
\end{align*}
where we denoted $E:=\|z(\theta, \cdot)\|_{L^{2}\left(\mathbb{R}^N\right)} + C\|u\|_{L^{2}\left(0, \theta ; L^{2}(\omega)\right)}$, and $C$ is a constant that varies from line to line. By \eqref{in1}, when the norm $\|u\|_{H^{1}\left(0, \theta ; L^{2}(\omega)\right)}$ is sufficiently small, we obtain
\begin{equation}
0 < E \le C_1\|u\|_{H^{1}\left(0, \theta ; L^{2}(\omega)\right)} <1 \label{eqt1}
\end{equation}
for some constant $C_1>0$. Using the inequality $\dfrac{\tau -1}{\log \tau} +\tau \leq -\dfrac{1+ \mathrm{e}^{-2}}{\log \tau}$ for $0<\tau <1$, with \eqref{in1} we obtain
\begin{align*}
\|u_0\|_{L^{2}\left(\mathbb{R}^N\right)} & \leq \frac{-C}{\log \left(C_1\|u\|_{H^1\left(0, \theta ; L^{2}(\omega)\right)}\right)}.
\end{align*}
\end{proof}

At this level, some comments should be made on the particular case of the heat equation (when $B=0$): Theorem \ref{thmlogstab1} enables us to derive a logarithmic stability result for a class of initial data of the heat equation posed in $\mathbb{R}^N$. A similar result was shown in \cite{XY'00} for the heat equation on a bounded domain $\Omega \subset \mathbb{R}^N$ with homogeneous Dirichlet boundary conditions on $\partial \Omega$.

Since the heat semigroup is analytic on $L^2\left(\mathbb{R}^N\right)$, we can show an improved logarithmic stability result for a large class of initial data
$$
\mathcal{I}_{\varepsilon, R}:=\left\{u_0 \in H^{2 \varepsilon}\left(\mathbb{R}^{N}\right) \colon \|u_0\|_{H^{2 \varepsilon}\left(\mathbb{R}^{N}\right)} \leq R\right\}
$$
for fixed $\varepsilon\in (0, 1)$ and $R > 0$. Although the used techniques are quite classical, we will give a full proof for the reader convenience.
\begin{theorem}\label{thmlogstab2}
Let $\theta>0$ be fixed and let $\omega \subset \mathbb{R}^N$ be a thick set. Assume also that $p \in \left(1,\dfrac{1}{1-\varepsilon}\right)$ and $s \in \left(0,1-\dfrac{1}{p}\right)$. Then there exists a positive constant $K(\varepsilon, R, \theta, \kappa_\theta,p,s)$ such that, for all $u_0 \in \mathcal{I}_{\varepsilon, R}$, we have
\begin{equation}
\|u_0\|_{L^{2}\left(\mathbb{R}^{N}\right)} \le K\left(\frac{\|u\|_{L^{2}(0,\theta ; L^2(\omega))}^{p}-1}{\log \|u\|_{L^{2}(0,\theta ; L^2(\omega))}}\right)^{\frac{s}{p}},
\end{equation}
where $u$ is the solution of system \eqref{e1} with $B=0$. Moreover, if $\|u\|_{L^{2}(0,\theta ; L^2(\omega))} <1$, then
\begin{equation}
\|u_0\|_{L^{2}\left(\mathbb{R}^{N}\right)} \le K\left(-\log \|u\|_{L^{2}(0,\theta ; L^2(\omega))}\right)^{-\frac{s}{p}}.
\end{equation}
\end{theorem}

\begin{proof}
Since $u_0 \in \mathcal{I}_{\varepsilon, R}$, we deduce by the logarithmic convexity that
\begin{equation*}
\|u(t)\|_{\left(\mathbb{R}^{N}\right)} \le R^{\frac{\theta-t}{\theta}}\left\|u\left(\theta\right)\right\|_{L^2\left(\mathbb{R}^{N}\right)}^{\frac{t}{\theta}} \quad\left(0 \le t \le \theta\right).
\end{equation*}
Assume $p > 1$. By integrating the above relation between $0$ and $\theta$, we obtain
$$\begin{aligned}
\int_{0}^{\theta}\|u(t)\|_{L^2\left(\mathbb{R}^{N}\right)}^{p} \d t & \le R^{p} \int_{0}^{\theta} \mathrm{e}^{p \frac{t}{\theta} \log \left(R^{-1}\left\|u\left(\theta\right)\right\|_{L^2\left(\mathbb{R}^{N}\right)}\right)} \d t \\
& \le \theta \frac{\left\|u\left(\theta\right)\right\|_{L^2\left(\mathbb{R}^{N}\right)}^{p}-R^{p}}{\log \left\|u\left(\theta\right)\right\|_{L^2\left(\mathbb{R}^{N}\right)}^{p}-\log R^{p}},
\end{aligned}$$
and thus
\begin{equation} \label{eqn1}
\|u\|_{L^{p}\left(0, \theta ; L^2\left(\mathbb{R}^{N}\right)\right)} \le \theta^{\frac{1}{p}}\left(\frac{\left\|u\left(\theta\right)\right\|_{L^2\left(\mathbb{R}^{N}\right)}^{p}-R^{p}}{\log \left\|u\left(\theta\right)\right\|_{L^2\left(\mathbb{R}^{N}\right)}^{p}-\log R^{p}}\right)^{\frac{1}{p}}. 
\end{equation}
On the other hand, by the semigroup representation of the solution $u$, we have
$$u_{t}=-(-\Delta)^{1-\varepsilon} \mathrm{e}^{t \Delta}(-\Delta)^{\varepsilon} u_{0}.$$
Moreover, since $\Delta$ is the generator of an analytic semigroup on $L^2\left(\mathbb{R}^{N}\right)$, we have
$$\left\|u_{t}(t)\right\|_{L^2\left(\mathbb{R}^{N}\right)} \le \frac{1}{t^{1-\varepsilon}}\left\|u_{0}\right\|_{H^{2\varepsilon}\left(\mathbb{R}^{N}\right)}.$$
We deduce from the above inequality that for $p \in \left(1,\dfrac{1}{1-\varepsilon}\right)$,
$$\left\|u_{t}\right\|_{L^{p}\left(0, \theta ; L^2\left(\mathbb{R}^{N}\right)\right)} \le R \frac{\theta^{\frac{1}{p}-(1-\varepsilon)}}{(1-p(1-\varepsilon))^{\frac{1}{p}}}.$$
Using $\|u(t)\|_{L^2\left(\mathbb{R}^{N}\right)} \le\|u(0)\|_{L^2\left(\mathbb{R}^{N}\right)} \le R$ for all $t$, we also have
\[
\|u\|_{L^{p}\left(0, \theta ; L^2\left(\mathbb{R}^{N}\right)\right)} \le R \theta^{\frac{1}{p}},
\]
and therefore
\begin{equation}
\|u\|_{W^{1, p}\left(0, \theta ; L^2\left(\mathbb{R}^{N}\right)\right)} \le R \theta^{\frac{1}{p}}\left(1+\frac{1}{\theta^{(1-\varepsilon)}(1-p(1-\varepsilon))^{\frac{1}{p}}}\right). \label{eqn2}
\end{equation}
Combining \eqref{eqn1} and \eqref{eqn2}, and using a Sobolev interpolation, we obtain for all $0<s<1$ that 
\[
\begin{aligned}
\|u\|_{W^{1-s, p}\left(0, \theta ; L^2\left(\mathbb{R}^{N}\right)\right)} \le C R^{1-s} \theta^{\frac{1}{p}}\left(1+\frac{1}{\theta^{(1-\varepsilon)}(1-p(1-\varepsilon))^{\frac{1}{p}}}\right)^{1-s} \\
\times\left(\frac{\left\|u\left(\theta\right)\right\|_{L^2\left(\mathbb{R}^{N}\right)}^{p}-R^{p}}{\log \left\|u\left(\theta\right)\right\|_{L^2\left(\mathbb{R}^{N}\right)}^{p}-\log R^{p}}\right)^{\frac{s}{p}}.
\end{aligned}
\]
Using the Sobolev embedding
\[
W^{1-s, p}\left(0, \theta ; L^2\left(\mathbb{R}^{N}\right)\right) \subset C\left(\left[0, \theta\right] ; L^2\left(\mathbb{R}^{N}\right)\right)
\]
for $(1-s) p>1$, there exists $K=K\left(R, \theta, p, \varepsilon, s\right)>0$ such that
\[
\|u(0)\|_{L^2\left(\mathbb{R}^{N}\right)} \le K\left(\frac{\left\|u\left(\theta\right)\right\|_{L^2\left(\mathbb{R}^{N}\right)}^{p}-R^{p}}{\log \left\|u\left(\theta\right)\right\|_{L^2\left(\mathbb{R}^{N}\right)}^{p}-\log R^{p}}\right)^{\frac{s}{p}}.
\]
Using the observability inequality \eqref{obsineq},
$$\|u(\theta)\|_{L^2\left(\mathbb{R}^{N}\right)} \leq \kappa_{\theta}\|u\|_{L^{2}(0, \theta ; L^2(\omega))},$$
we deduce that
\[
\|u(0)\|_{L^2\left(\mathbb{R}^{N}\right)} \le K\left(\frac{\kappa_{\theta}^{p}\|u\|_{L^{2}(0,\theta ; L^2(\omega))}^{p}-R^{p}}{\log \kappa_{\theta}^{p}\|u\|_{L^{2}(0,\theta ; L^2(\omega))}^{p}-\log R^{p}}\right)^{\frac{s}{p}}.
\]
Then considering both cases $\kappa_{\theta} \ge R$, $\kappa_{\theta} < R$, and using again the concavity of the logarithm function, we deduce that
$$\|u(0)\|_{L^2\left(\mathbb{R}^{N}\right)} \le K\left(\frac{\|u\|_{L^{2}(0,\theta ; L^2(\omega))}^{p}-1}{\log \|u\|_{L^{2}(0,\theta ; L^2(\omega))}}\right)^{\frac{s}{p}}.$$
This ends the proof.
\end{proof}

\begin{remark}
We emphasize that in Theorem \ref{thmlogstab2} we have considered observation sets $\omega \subset \mathbb{R}^N$ that are thick. Such a result improves \cite[Theorem 1.1]{ACCOP'20} in terms of observation where the authors consider open sets $\omega$ such that $\mathbb{R}^N\setminus \omega$ is bounded, which is far to be optimal.
\end{remark}

\section{Comments and open problems} \label{sec5}
We have investigated the interplay between observability and inverse problems for a class of equations which represents a prototype of parabolic equations with unbounded coefficients (of gradient type) via the logarithmic convexity estimate. More precisely, we have proven that the observability inequality along with the logarithmic convexity imply conditional logarithmic stability of initial data.

In this paper, we have considered a case including a non-analytic semigroup which has not been studied in the literature within this context. Also, stability estimates for initial data for parabolic equations with unbounded coefficients have not been considered before up to our knowledge, except in \cite{ACM'22} where analyticity have played a crucial role.

In the particular (but important) case of heat equation, i.e. the case without the drift term, that is $B=0$, we can improve some interesting results obtained in the recent paper \cite{ACCOP'20, ACCOP'23} in terms of the observability regions $\omega\subset \mathbb{R}^N$ for the reconstruction of the initial temperatures. More precisely, we can consider sharp observation regions $\omega$ given by thick sets instead of considering sets such that the unobserved region $\mathbb{R}^N\setminus \omega$ is bounded. Furthermore, for the general case $B\neq 0$, the results can be improved as far as the observation region $\omega$ is sharpened.

In the present paper, we have proven a logarithmic convexity estimate for the Ornstein-Uhlenbeck equation leveraging the explicit representation formula of the corresponding semigroup. This raises the following problems:
\begin{enumerate}
\item Can one prove a logarithmic convexity estimate in the absence of an explicit formula for the associated semigroup as in the analytic case? For instance, can we consider more general Ornstein-Uhlenbeck operators as $\mathcal{A}=\dv(Q \nabla)- R x\cdot x + B x \cdot \nabla,$
where $Q=Q^*$ and $R=R^*$ are real constant $N\times N$-matrices which are positive semidefinite? Under suitable assumptions on the matrices, $B, Q$ and $R$, the observability inequality holds true, see \cite[Subsection 6.4]{DSV'22}. Note that this case also covers some situations where the observation region $\omega$ is not thick.
\item Given the perturbation argument in Section \ref{sec2}, it is natural to ask whether the logarithmic convexity is preserved when a main abstract operator $\mathcal{A}$ is perturbed by an operator $\mathcal{B}$. This is true for instance in the analytic case, when $\mathcal{A}$ is perturbed by a bounded operator or more generally a relatively $\mathcal{A}$-bounded operator, since analyticity is preserved in this case, see e.g. \cite[Theorem 2.10, p.~176]{EN'99}.
\end{enumerate}

\section*{Acknowledgment}
The first named author would like to thank Giorgio Metafune for a fruitful discussion with invaluable comments.

\bibliographystyle{unsrt}

\end{document}